\documentclass[12pt]{article}
\usepackage{amssymb}
\usepackage{amsthm}
\usepackage{amsthm,amsmath,indentfirst,amssymb}
\usepackage{pstricks}
\usepackage{algorithm} %format of the algorithm
\usepackage{algorithmic} %format of the algorithm
\usepackage{blindtext}
\usepackage{float}
\usepackage{graphicx}
\usepackage{cite}
\makeatletter

\newcommand{\Rmnum}[1]{\expandafter\@slowromancap\romannumeral #1@}
\makeatother
\newtheorem{theorem}{Theorem}[section]
\newtheorem{lemma}[theorem]{Lemma}
\newtheorem{corollary}[theorem]{Corollary}
\newtheorem{conjecture}[theorem]{Conjecture}

\begin{document}
\title{Feedback vertex number of Sierpi\'{n}ski-type graphs\thanks{Research supported by NSFC (No. 11571294)}}
\author{Lili Yuan, Baoyindureng Wu\footnote{Corresponding author.
Email: baoywu@163.com (B. Wu) }, Biao Zhao\\
\small  College of Mathematics and System Sciences, \\
\small Xinjiang University, Urumqi, 830046, Xinjiang, China\\}
%\renewcommand{\thefootnote}{\fnsymbol{footnote}}
%\footnotetext[1] {This was supported by NSFC (No. 11571294).}
%\par {Biography:}
%\par {Corresponding author:}  Baoyindureng Wu, E-mail: baoywu@163.com}

\date{}
\maketitle \baselineskip 18pt
%\maketitle
\textbf{Abstract}: The feedback vertex number $\tau(G)$ of a graph $G$ is the minimum number of vertices that can be deleted from $G$ such that the resultant graph does not contain a cycle. We show that $\tau(S_p^n)=p^{n-1}(p-2)$ for the Sierpi\'{n}ski graph $S_p^n$ with $p\geq 2$ and $n\geq 1$. The generalized Sierpi\'{n}ski triangle graph $\hat{S_p^n}$ is obtained by contracting all non-clique edges from the Sierpi\'{n}ski graph $S_p^{n+1}$. We prove that $\tau(\hat{S}_3^n)=\frac {3^n+1} 2=\frac{|V(\hat{S}_3^n)|} 3$, and give an upper bound for $\tau(\hat{S}_p^n)$ for the case when $p\geq 4$. \\


\section{Introduction}

In this paper, we consider simple, finite, undirected graphs only, and refer to \cite{Bondy}
for undefined terminology and notation. For a graph $G=(V(G),E(G))$, the {\it order} and {\it size} are $|V(G)|$ and $|E(G)|$, respectively. We denote the order and the size of $G$ simply by $|G|$ and $||G||$, respectively. For a vertex $v$, the {\it degree}
of a vertex $v$, denoted by $d_G(v)$, is the number of edges which
are incident with $v$ in $G$, and the neighborhood of $v$, denoted by $N_G(v)$, is the set of the vertices adjacent to $v$ in $G$.

For a set $X\subseteq V,$ let $G-X$ be the graph obtained from $G$ by removing vertices of $X$ and all the edges that are incident to a vertex of $X$. The subgraph $G-(V(G)\setminus X)$ is said to be the {\it induced subgraph} of $G$ induced by $X$, and is denoted by $G[X]$.
We call $X$ is a {\it feedback vertex set} of $G$ if $G-X$ is a forest. The {\it feedback vertex number} of $G$, denoted by $\tau(G)$, is the cardinality of a  minimum feedback vertex set of $G$. The order of a maximum induced forest of $G$ is denoted by $f(G)$. It is clear that $\tau(G)+f(G)=|G|$.

So, determining the feedback number of a graph $G$ is equivalent to finding the maximum induced forest of $G$, proposed first by Erd\H{o}s, Saks and S\'{o}s \cite{Erdos}. Some results on the maximum induced forest are obtained for several families of graphs, such as planar with high girth \cite{Dross, Kelly}, outerplanar graphs \cite{Hosono, Pel}, triangle-free planar graphs \cite{Sala}.  The problem of determining the feedback vertex number has been proved to be NP-complete for general graphs \cite{Karp}. However, the problem has been studied for some special graphs, such as hypercubes \cite{Focardi, Pike}, toroids \cite{Made}, de Bruijn graphs \cite{Xu11}, de Bruijn digraphs \cite{Xu10}, 2-degenerate graphs \cite{Boro}.
A review of results and open problems on the feedback vertex number was given by Bau and Beineke \cite{Bau}.

Graphs whose drawings can be viewed as approximations to the famous Sierpi\'{n}ski triangle have been studied extensively in the past 25 years, see a recent survey \cite{Hinz17} for the collection of the results and the related works. We follow the notation there.
For an integer $k$, let $\mathbb{N}_k$ be the set of all integer not less than $k$. In particular, $\mathbb{N}_1$ is denoted simply by $\mathbb{N}$.  For an integer $p\in  \mathbb{N}$, let $P=\{0, 1,\ldots, p-1\}$. For convenience, let $[n]=\{1,\ldots, n\}$ for an integer $n\in \mathbb{N}$. Let $P^{n}$ be the set of all $n$-tuples on $P$.

\vspace{3mm} \noindent{\bf Definition 1.} (\cite{Hinz17}) For two integers $p\geq 1$ and $n\geq 0$, the Sierpi\'{n}ski  graph $S_p^n$ is given by
\begin{equation*}V(S_p^n)=P^n,  E(S_p^n)=\{\{\underline{s}ij^{d-1}, \underline{s}ji^{d-1}\}|\ i,j\in P, i\neq j;d\in [n];\underline{s}\in P^{n-d}\}.
\end{equation*}

A vertex $s_1s_2\ldots s_n$ of $S_p^n$ is called an extreme vertex if $s_1=\cdots=s_n$. There are precisely $p$ extreme vertices in $S_p^n$. It is easy to see that $d_{S_p^n}(u)=p-1$ for an extreme vertex $u$ and $d_{S_p^n}(v)=p$ for all other vertices $v$.
In the trivial cases $n=0$ or $p=1$, there is only one vertex and no edge, i.e. $S_p^0\cong K_1\cong S_1^n$; moreover, $S_p^1\cong K_p$, where $K_p$ denotes the complete graph of order $p$. The first interesting case is $p=2$, where $S_2^n\cong P_{2^n}$.
The drawings of the Sierpi\'{n}ski graphs $S_3^3$ and $S_5^2$ are shown in Fig. 1.

\begin{figure}[!h]
\begin{center}
\scalebox {0.6}[0.6]{\includegraphics {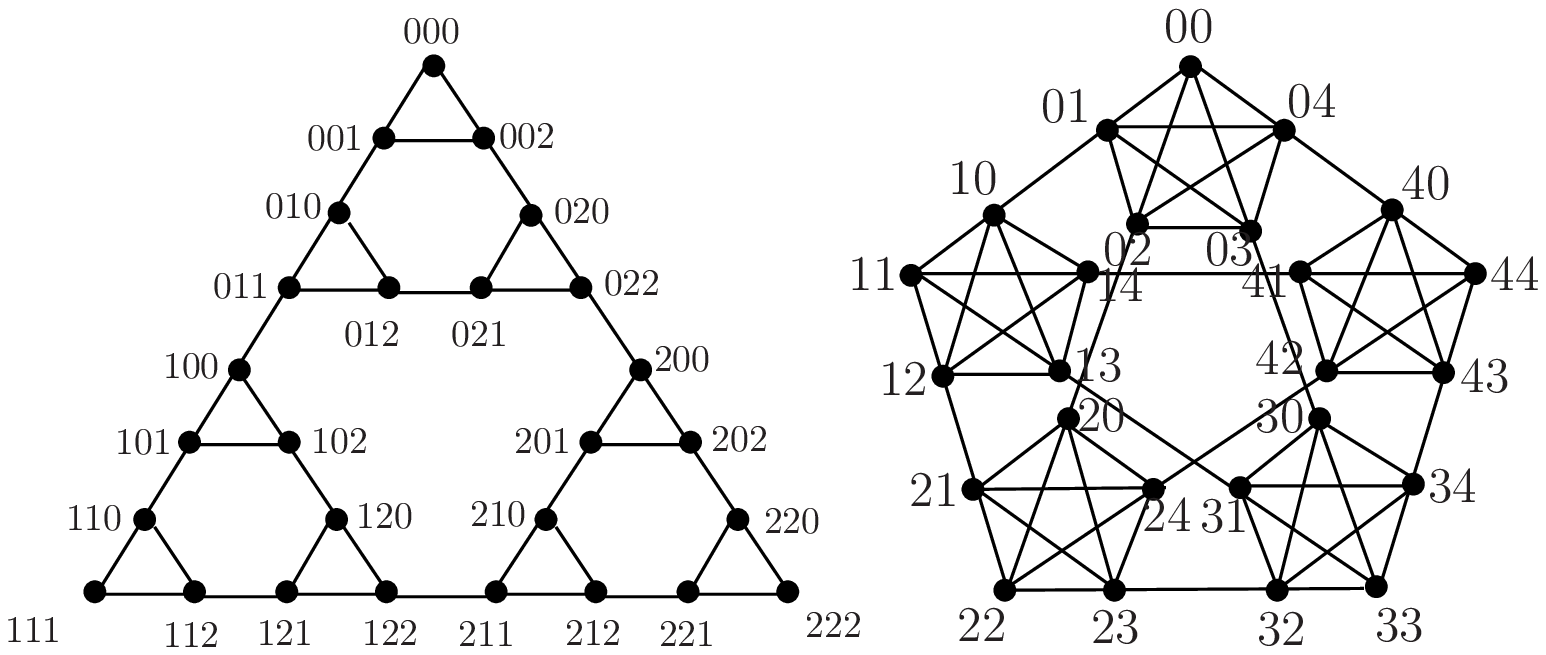}}
%\vskip .3cm
\centerline{Fig. 1\hspace{3mm} $S_3^3$ and $S_5^2$}
\end{center}
\end{figure}

\noindent{\bf Definition 2.} (\cite{Hinz17}) For two integers $p\geq 1$ and $n\geq 0$, the generalized Sierpi\'{n}ski triangle graph $\hat{S_p^n}$ is obtained by contracting all non-clique edges (with the form $\{\underline{s}ij^l, \underline{s}ji^l\}$ for distinct $i, j\in P$ and $l\in [n]$, $\underline{s}\in P^{n-l}$) from the Sierpi\'{n}ski graph $S_p^{n+1}$,

\vspace{3mm} Sierpi\'{n}ski-type graphs have many interesting properties, see \cite{Hinz02} for its plnarity, \cite{Klav05} for crossing number,  \cite{Hinz12, Lin, Klav08, Xue12} for its various coloring, \cite{Hinz90, Xue14} for various distance, \cite{Lin13} for the global strong defensive alliances, \cite{Lin11} for the hub number, \cite{Teufl10,Teufl11} for enumeration of matchings and spanning trees.

\vspace{3mm} \vspace{3mm} The main topic of the paper focus on the feedback vertex number of Sierpi\'{n}ski-type graphs. In Section 2, we show that $\tau(S_p^n)=p^{n-1}(p-2)$. In Section 3, we show that $\tau(\hat{S}_3^n)=\frac {3^n+1} 2=\frac{|\hat{S}_3^n|} 3.$
\vspace{3mm} In Section 4, we present an upper bound for the feedback vertex number of $\hat{S}_p^n$ for any $p\geq 4$. Some other relevant results are presented as well.

\section{\large Sierpi\'{n}ski  graph }

In this section, we determine the exact value of feedback number of Sierpi\'{n}ski graph $S_p^n$. For two vertex subsets $X, Y$ of a graph $G$, we denote by $E_G[X, Y]$ the set of edges of $G$ with one end in $X$ and the other end in $Y$.
First we show the following.

\begin{lemma} For two integers $n\geq 1$ and $p\geq 2$, $\tau(S_p^n)\geq p^{n-1}(p-2)$.
\end{lemma}

\begin{proof} By induction on $n$. If $n=1$, $S_p^n\cong K_p$. Since $$\tau(S_p^n)=\tau(K_p)=p-2= p^0(p-2)=p^{n-1}(p-2),$$ the result then follows.
Next assume that $n\geq 2$ and $X$ is a minimum feedback vertex set of $S_p^n$.
Let $V_i=\{i\underline{k}|\ \underline{k}\in V(S_p^{n-1})\}$ for each $i\in P$. From the definition of $S_p^n$, we know that  $S_p^n[V_i]\cong S_p^{n-1}$. Let $X_i=X\cap V_i$ for each $i\in P$.
Since $X$ is a minimum feedback vertex set of $S_p^n$, $X_i$ is a feedback vertex set of $S_p^n[V_i]$. By induction hypothesis, $|X_i|\geq p^{n-2}(p-2)$. Thus, $|X|\geq pp^{n-2}(p-2)=p^{n-1}(p-2)$.
\end{proof}

A 2-element set $Y\subseteq P^{m-1}$ for an integer $m\in \{2, \ldots, n\}$ is said to be {\it pairable} if $Y=\{\underline{s}a, \underline{s}b\}$, where $\underline{s}\in P^{m-2}$ and $\{a,b\}\in \big(^P_2\big)$. In this case, let $head(Y)=\underline{s}$. Further, we define $$C(Y):=C_1(Y)\cup C_2(Y),$$ where $C_1(Y)=\{\underline{s}aa, \underline{s}ab, \underline{s}ba, \underline{s}bb\}$ and $C_2(Y)=\{\underline{s}k(k-1), \underline{s}k(k+1)|\ k\in P\setminus\{a,b\}\},$ and $k-1, k+1$ are taken modulo $p$.
In general, a set $Y\subseteq P^{m-1}$ is said to be {\it pairable} if $Y$ can be partitioned into some 2-element sets $Y_1, \ldots, Y_q$ such that $Y_i$ is pairable for each $i$, and $head(Y_i)\neq head(Y_j)$ for any $i, j$ with $i\neq j$.
Moreover, let us define $$head(Y):=\bigcup_{i\in [q]} head(Y_{i})\ \text{and } C(Y):=\bigcup_{i\in [q]} C(Y_{i}).$$

Observe that if $Y$ is pairable (since $head(Y_i)\neq head(Y_j)$ for any $i, j$ with $i\neq j$), then it must be partitioned into some 2-element pairable sets in a unique way. So, we call $\{Y_1, \ldots, Y_q\}$ the pairable partition of $Y$, and call each $Y_i$ a block of $Y$. Take $Y=\{0, 2\}\subseteq V(S^1_5)$ for an example. Then $C(Y)=\{00,02,20,22,10,12,32,34,43,40\} \subseteq V(S^2_5),$
see the subgraphs $S^1_5[Y]$ and $S^2_5[C(Y)]$ as depicted in red in Fig 2.

\begin{figure}[!h]
\begin{center}
\scalebox {0.7}[0.7]{\includegraphics {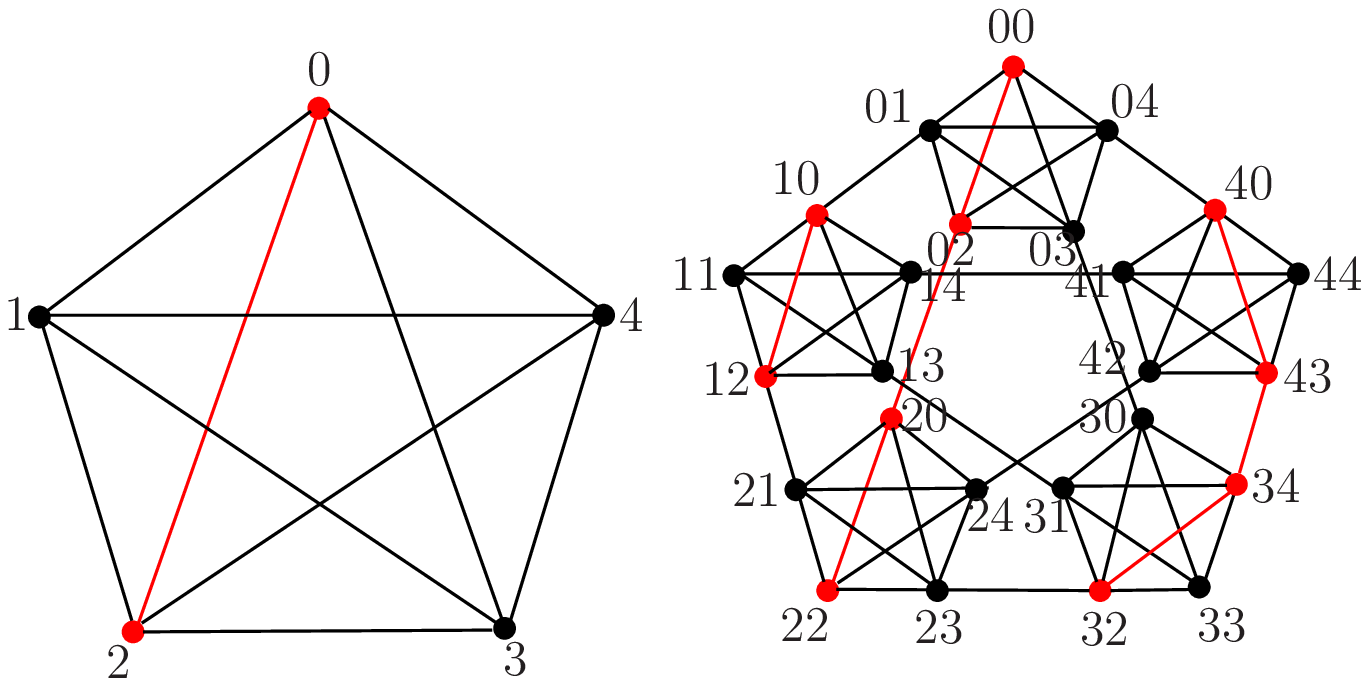}}
%\vskip -.3cm
\centerline{Fig. 2. $S^1_5[Y]$ and $S^2_5[C(Y)]$, where $Y=\{0, 2\}$ }
\end{center}
\end{figure}

\begin{lemma} If $Y\subseteq P^{m-1}$ is pairable for an integer $m\geq 2$, then
$C(Y)$ is pairable, and $head(C(Y))=\{\underline{s}a|\ \underline{s}\in head(Y), a\in P\}$.
\end{lemma}

\begin{proof}
Let $\{Y_1, \ldots, Y_q\}$ be the pairable partition of $Y$, where for each $i$, $Y_i=\{\underline{s_i}a_i, \underline{s_i}b_i\}$ for some $\underline{s_i}\in P^{m-2}$ and $\{a_i,b_i\}\in \big(^P_2\big)$.
By our definition, $$C(Y):=\bigcup_{i\in [q]} C(Y_{i}),$$ where
$C(Y_i):=C_1(Y_i)\cup C_2(Y_i),$ and $C_1(Y_i)=\{\underline{s_i}a_ia_i, \underline{s_i}a_ib_i, \underline{s_i}b_ia_i, \underline{s_i}b_ib_i\}$ and $C_2(Y_i)=\{\underline{s_i}k(k-1), \underline{s_i}k(k+1)|\ k\in P\setminus\{a_i,b_i\}$, where $k-1, k+1$ \text{are taken modulo $p$}$\}$.

Note that for an integer $i\in [q]$, $C_1(Y_i)$ can be partitioned into two 2-element pairable sets $\{\underline{s_i}a_ia_i, \underline{s_i}a_ib_i\}$ and $\{\underline{s_i}b_ia_i, \underline{s_i}b_ib_i\}$, and $C_2(Y_i)$ can be partitioned into $p-2$ 2-element pairable sets $\{\underline{s_i}k(k-1), \underline{s_i}k(k+1)\}$, where $k\in P\setminus\{a_i,b_i\}$. It is clear that the set of the heads of these $p$ sets is $\{\underline{s_i}a|\ a\in P\}$. Moreover, by our assumption, since $\underline{s_i}\neq \underline{s_j}$ for distinct integers $i, j$, $\underline{s_i}a\neq \underline{s_j}b$ for any $a, b \in P$ (even when $a=b$). This proves the lemma.
\end{proof}

In what follows, for an integer $m\geq 2$ and a pairable set $Y\subseteq P$, let $C^{m}(Y)=C(C^{m-1}(Y))$.

%\begin{lemma}
%Let $Y\subseteq P^{m-1}$ for an integer $m\in \{2, \ldots, n\}$ a pairable set with $Y=\{\underline{s}a, \underline{s}b\}$, where $\underline{s}\in P^{m-2}$ and $\{a,b\}\in \big(^P_2\big)$.
%\end{lemma}
%
%\begin{lemma}
%\end{lemma}

\begin{lemma} Let $Y_1=\{1, 2\}\subseteq V(S_p^1)$, and $Y_m=C^{m-1}(Y_1)$ for $m\in [n]\setminus \{1\}$. Then $Y_m$ is a pairable set with $head(Y_m)=P^{m-1}$, $|Y_m|=2p^{m-1}$ and $S_p^m[Y_m]$ is a forest for each $m\in [n]$.
\end{lemma}
\begin{proof}

First of all, since $Y_1=\{1,2\}$ is a pairable set, by Lemma 2.2, $Y_2$ is pairable.
By the definition, $$Y_2=C(Y_1)=\{11, 12, 21, 22\}\cup\{ k(k-1), k(k+1)\mid k\in P\setminus\{1, 2\}\}.$$ Note that $head(Y_2)=\{0, 1, \dots, p-1\}=P$ and $|Y_2|=2p$. One can see that $S_p^2[Y_2]$ is a linear forest consisting of two paths $$11, 12, 21, 22\ \text{and}\ 32, 34, 43,45, \cdots, (p-1)(p-2), (p-1)0, 0(p-1), 01.$$

Assume that $m\geq 3$ and the result is true for smaller value of $m$.
By induction hypothesis, $Y_{m-1}$ is a pairable set
with $head(Y_{m-1})=P^{m-2}$ and $|Y_{m-1}|=2p^{m-2}$. Since $Y_m=C(Y_{m-1})$, by Lemma 2.2,
$Y^m$ is pairable, and $$head(Y_m)=\{sa|\ a\in head(Y_{m-1}), a\in P\}=P^{m-1}.$$ Thus $|Y_m|=2 head(Y_m)=2|P^{m-1}|=2p^{m-1}$.

\vspace{3mm} Next we show that $S_p^m[Y_m]$ is a forest of $S_p^m$. Let $Y_{m}=Y_{m1}\cup Y_{m2}$, where $Y_{m1}=C_1(Y_{m-1}), Y_{m2}=C_2(Y_{m-1})$.

\vspace{3mm}\noindent{\bf Claim 1.} $E_{S_p^m}[Y_{m1}, Y_{m2}]=\emptyset.$

\begin{proof} Suppose not, and let $y_1\in Y_{m1}$ and $y_2\in Y_{m2}$ with $y_1y_2\in E(S_p^m)$.
By the definition of $Y_{m1}$ and $Y_{m2}$, there exist blocks $\{\underline{s_1}a, \underline{s_1}b\}$ and $\{\underline{s_2}c, \underline{s_2}d\}$ of $Y_{m-1}$ such that $y_1\in C_1(\{\underline{s_1}a, \underline{s_1}b\})$ and $y_2\in C_2(\{\underline{s_2}c, \underline{s_2}d\})$, where $\underline{s_1}, \underline{s_2}\in P^{m-2}$ and $a, b, c, d \in P$. Recall that
$C_1(\{\underline{s_1}a, \underline{s_1}b\})=\{\underline{s_1}aa, \underline{s_1}ab, \underline{s_1}ba, \underline{s_1}bb\}$ and $C_2(\{\underline{s_2}c, \underline{s_2}d\})=\{\underline{s_2}k(k-1), \underline{s_2}k(k+1)|\ k\in P\setminus \{c, d\}\}$.

\vspace{3mm}\noindent{\bf Case 1.} $y_2=\underline{s_2}k(k-1)$ for an element $k\in P\setminus \{c, d\}$.

\vspace{1mm} Since $N_{S_p^m}(y_2)=\{\underline{s_2}(k-1)k\}\cup \{\underline{s_2}kc|\ c\in P\setminus \{k-1\}\}$,
$y_1\in \{\underline{s_1}aa, \underline{s_1}ab, \underline{s_1}ba, \underline{s_1}bb\}$, and $y_1y_2\in E(S_p^m)$, we have $(\{\underline{s_2}(k-1)k\}\cup \{\underline{s_2}kc|\ c\in P\setminus \{k-1\}\})\cap \{\underline{s_1}aa, \underline{s_1}ab, \underline{s_1}ba, \underline{s_1}bb\}\neq \emptyset$. It follows that $\underline{s_1}=\underline{s_2}$. Moreover, since  $Y_{m-1}$ is pairable, the heads of different blocks are distinct. So,  $\{\underline{s_1}a, \underline{s_1}b\}=\{\underline{s_2}c, \underline{s_2}d\}$ and thus $\{a, b\}=\{c, d\}$.

Now rewrite $N_{S_p^m}(y_2)=\{\underline{s_1}(k-1)k\}\cup \{\underline{s_1}kc|\ c\in P\setminus \{k-1\}\}$, where $k\in P\setminus \{a, b\}$. However, since $k\in P\setminus \{a, b\}$, it is clear that $N_{S_p^m}(y_2)\cap \{\underline{s_1}aa, \underline{s_1}ab, \underline{s_1}ba, \underline{s_1}bb\}=\emptyset$, a contradiction, implying $y_1y_2\notin E(S_p^m)$.

\vspace{3mm}\noindent{\bf Case 2.} $y_2=\underline{s_2}k(k+1)$ for an element $k\in P\setminus \{c, d\}$.

\vspace{1mm} Proof of this case is completely similar to that of Case 1.

\end{proof}

\vspace{3mm}\noindent{\bf Claim 2.} $S_p^m[Y_{m2}]$ is a forest.

\begin{proof} Recall that $Y_{m2}=C_2(Y_{m-1})$ and $C_2(\{\underline{s_i}a_i,\underline{s_i}b_i\})=\{\underline{s_i}k(k+1),\underline{s_i}k(k-1)\mid k\in P\setminus \{a_i, b_i\}\}$ for a block $\{\underline{s_i}a_i, \underline{s_i}b_i\}$.
One can easily check that $S_p^m[C_2(\{\underline{s_i}a_i, \underline{s_i}b_i\})]$ is a linear forest with at most
two components. If $\{\underline{s_j}a_j, \underline{s_j}b_j\}$ is another block of $Y_{m-1}$, then by $\underline{s_i}, \underline{s_j}\in P^{m-2}$ with $\underline{s_i}\neq  \underline{s_j}$,  $E_{S_p^m}[\{\underline{s_i}k(k-1),\underline{s_i}k(k+1)\}, \{\underline{s_j}l(l-1),\underline{s_j}l(l+1)\}]=\emptyset$ for any $k\in P\setminus \{a, b\}$ and $l\in P\setminus \{c, d\}$. So,
$E_{S_p^m}[C_2(s_ia_i,s_ib_i), C_2(\underline{s_j}a_j,\underline{s_j}b_j) ]=\emptyset$. Thus $S_p^m[Y_{m2}]$ is a forest.
\end{proof}

\vspace{3mm}\noindent{\bf Claim 3.} $S_p^m[Y_{m1}]$ is a forest.

\vspace{3mm} Suppose that $C$ is an induced cycle of $S_p^m[Y_{m1}]$, and let $S=\{\underline{s}|\ \underline{s}\in P^{m-2}$ and there is a vertex $v\in V(C)$ such that $v\in C_1(\{\underline{s}a, \underline{s}b\})\}$.  Since $C_1(\{\underline{s}a, \underline{s}b\})=\{\underline{s}aa, \underline{s}ab, \underline{s}ba, \underline{s}bb\}$ for a block $\{\underline{s}a,\underline{s}b\}$  of $Y_{m-1}$, it is clear that $S_p^m[C_1(\{\underline{s}a, \underline{s}b\})]\cong P_4$.
It is not hard to see that $V(C)= \bigcup_{\underline{s}\in S} C_1(\{\underline{s}a, \underline{s}b\})$.
Let $C'=S_p^{m-1}[\cup_{\underline{s}\in S}\{\underline{s}a, \underline{s}b\} ]$. One can see that $C'$ is a cycle of $S_p^{m-1}[Y_{{m-1}}]$, a contradiction.
This proves that $S_p^{m}[Y_{m1}]$ contains no cycle.

\vspace{3mm} Summing up Claims 1-3, we conclude that $S_p^{m}[Y_m]$ is a forest.
\end{proof}

\begin{theorem} For any two integer $p\geq 2$, $n\geq 1$, $\tau(S_p^n)=p^{n-1}(p-2).$
\end{theorem}

\begin{proof} By Lemma 2.3, $f(S_p^n)\geq 2p^{n-1}$, and thus $\tau(S_p^n)\leq p^n-f(S_p^n )\leq p^{n-1}(p-2).$ On the other hand, by Lemma 2.1, $\tau(S_p^n)\geq p^{n-1}(p-2).$ The result then follows.
\end{proof}

In \cite{Hinz17}, two kinds of regularization of Sierpi\'{n}ski graphs were introduced.

\vspace{3mm}\noindent {\bf Definition 3.} (\cite{Hinz17})  For $p\in N$ and $n\in N$, the graph $^+S_p^n$ is defined by
\begin{equation*}V(^+S_p^n)=P^n\cup\{w\},  E(^+S_p^n)=E(S_p^n)\cup\{\{w,i^n\}|\ i\in P\}.\end{equation*}

\begin{lemma} (\cite{Hinz17}) If $p\in N$ and $n\in N$, then $|^+S_p^n|=p^n+1$ and $\|^{+}S_p^n\|=\frac{p}{2}(p^n+1).$
\end{lemma}

\vspace{3mm}\noindent {\bf Definition 4.} (\cite{Hinz17}) For $p\in N$ and $n\in N$, the graph $^{++}S_p^n$ is defined by
\begin{equation*}V(^{++}S_p^n)=P^n\cup\{p\bar{s}|\ \bar{s}\in P^{n-1}\}, \qquad \qquad \qquad \qquad \qquad \qquad \qquad \qquad \end{equation*}
\begin{equation*} E(^{++}S_p^n)=E(S_p^n)\cup\{\{p\bar{s},p\bar{t}\}|\ \{\bar{s}, \bar{t}\}\in E(S_p^{n-1})\}\cup\{\{pi^{n-1},i^n\}|\ i\in P\}.\end{equation*}

\begin{lemma}(\cite{Hinz17})  If $p\in N$ and $n\in N,$ then $|^{++}S_p^n|=(p+1)p^{n-1}$ and $\|^{++}S_p^n\|=\frac{p+1}{2}p^n.$
\end{lemma}

Drawings of $^+S^2_4$ and $^{++}S^2_4$ are shown in Fig. 3, where the left one is $^+S^2_4$ and the other one is $^{++}S^2_4$.

\begin{figure}[!h]
\begin{center}
\scalebox {0.6}[0.6]{\includegraphics {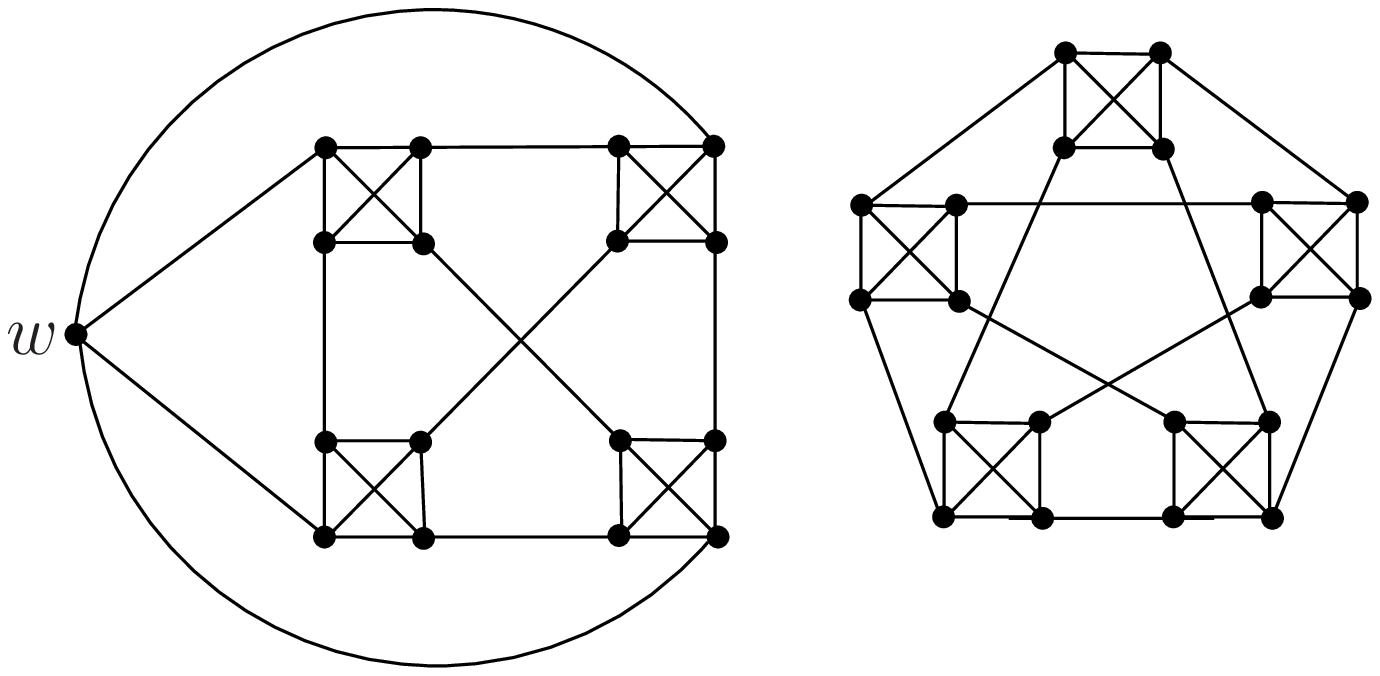}}
\centerline{Fig. 3\quad $^+S^2_4$ and $^{++}S^2_4$}
\end{center}
\end{figure}

\begin{corollary} For two integers $p\geq 2$ and $n\geq 1$, $$\tau(^+S_p^n)=\left \{
\begin{array}{ll}
p^{n-1}(p-2), & \mbox{if $p\geq 3$ } \\
1, & \mbox{if $p=2$}.
\end{array}
\right.
$$
\end{corollary}
\begin{proof} If $p=2$, $^+S_p^n$ is a cycle, and thus $\tau(^+S_p^n)=1$. Now let us consider the case when $p\geq 3$. Since $S_p^n$ is a subgraph of $^+S_p^n$, $\tau(^+S_p^n)\geq \tau(S_p^n)=p^{n-1}(p-2)$.
On the other hand, let $Y_n$ be the set as in the proof of Lemma 2.3.
Let $Y_n'=(Y_n\setminus \{1^n\})\cup \{1^{n-1}0, w\}$. It can be checked that $^+S_p^n[Y_n']$ is a forest of $^+S_p^n$ with  $|Y_n'|=2p^{n-1}+1$. Thus $\tau(^+S_p^n)\leq p^{n-1}(p-2)$, the result follows.
\end{proof}

\begin{corollary} For two integers $p\geq 2$ and $n\geq 1$, $\tau(^{++}S_p^n)=\tau(S_p^n)+\tau(S_p^{n-1}).$
\end{corollary}
\begin{proof} By the definition of $^{++}S_p^n$, clearly, $\tau(^{++}S_p^n)\geq \tau(S_p^n)+\tau(S_p^{n-1}).$
Next, we prove $\tau(^{++}S_p^n)\leq \tau(S_p^n)+\tau(S_p^{n-1}).$
Let $Y_n=C^{n-1}(Y_1)$ be the set defined as in the proof of Lemma 2.3, where $Y_1=\{1,2\}$. On the other hand,
we choose $Y_1^*=\{3,4\}$, and $Y_{n-1}^*=C^{n-2}(Y_1^*)$ as defined in the proof of Lemma 2.3. By the proof of Lemma 2.3, $^{++}S_p^n[Y_n]\cong S_p^n[Y_n]$ is a forest, and $^{++}S_p^n[pY_{n-1}^*]\cong S_p^n[Y_{n-1}^*]\cong S_p^n[Y_{n-1}]$ is a forest.
Moreover, $E_{^{++}S_p^n}[Y_n, pY_{n-1}^*]=\emptyset$. Thus $^{++}S_p^n[Y_n\cup pY_{n-1}^*]$ has no cycles, and $\tau(^{++}S_p^n)\leq \tau(S_p^n)+\tau(S_p^{n-1}).$
\end{proof}

\section{\large Sierpi\'{n}ski triangle graph}

A class of graphs that often has been mistaken for and also been called Sierpi\'{n}ski  graphs can be obtained from the latter by simply contracting all non-clique edges. We will call them Sierpi\'{n}ski triangle graphs.
Let $T:=\{0, 1, 2\}$ and $\hat{T}:=\{\hat{0},\hat{1},\hat{2}\}$.
Denote the vertex obtained by contracting the edge $\{\underline{s}ij^{n-v+1}, \underline{s}j i^{n-v+1}\}$ by $\underline{s}\{i, j \}$. Then the vertex set of $\hat{S_3^{n}}$ can be written as $$V(\hat{S_3^{n}})=\hat{T}\cup \{\underline{s}\{i,j\}|\ \underline{s}\in T^{v-1}, v\in [n], \{i,j\}\in \big(^T_2 \big)
\}.$$
Further, we simply replace each vertex $\underline{s}\{i, j \}$ by $\underline{s}k$, where $k=3-i-j$. So,
\begin{eqnarray*}
E(\hat{S_3^{n+1}})
&=&\{\{\hat{k},k^nj\}|\ k\in T,j\in T\setminus\{k\}\}\\
&\cup&\{\{\underline{s}k,\underline{s}j\}|\ \underline{s}\in T^{n},\{j,k\}\in \big(^T_2 \big)\}\\
&\cup& \{\{\underline{s}(3-i-j)i^{n-v}k,\underline{s}j\}|\ \underline{s}\in T^{v-1},v\in[n],i\in T,j,k\in T\setminus\{i\}\}.
\end{eqnarray*}

The Sierpi\'{n}ski triangle graph $\hat{S_3^3}$ is shown in Fig. 4.

\begin{figure}[!h]
\begin{center}
\scalebox {0.7}[0.7]{\includegraphics {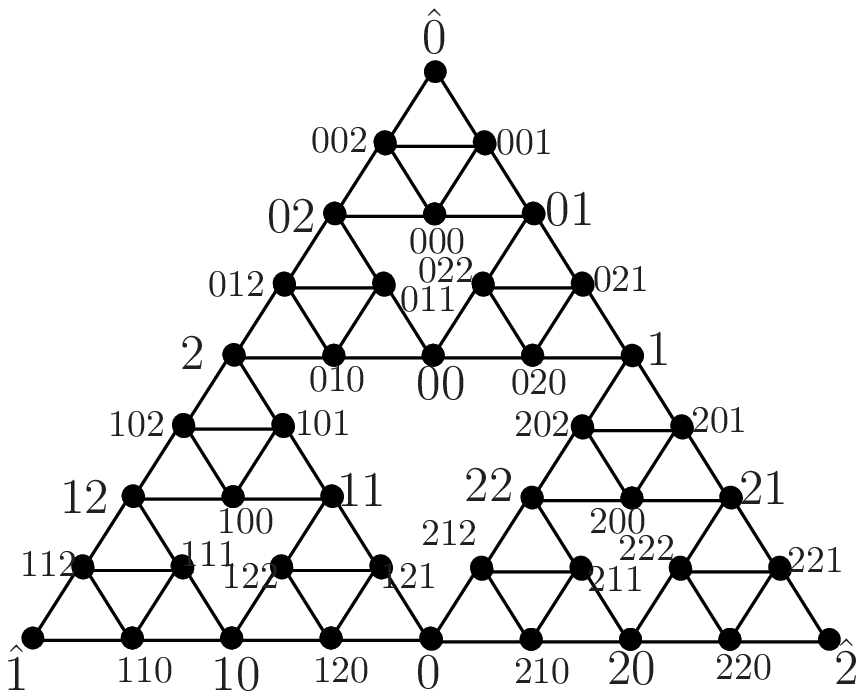} }
\centerline{Fig. 4\quad $\hat{S}_3^3$}
\end{center}
\end{figure}

\begin{lemma} (\cite{Jak})
If $p\in N$ and $n\in N_0
$, then $|\hat{S}^n_p|=\frac p 2(p^n+1)$ and $||\hat{S}^n_p||=\frac {p-1} 2 p^{n+1}$.
\end{lemma}

\begin{theorem}
(i) For any $n\geq 0$,
each minimum feedback vertex set of $\hat{S}_3^n$ contains at most one vertex in $\{\hat{0}, \hat{1}, \hat{2}\}$,

(ii) for $n\geq 0$, $\hat{S}_3^n$ has a minimum feedback vertex set $A_n$ such that $|A_n\cap \{\hat{0}, \hat{1}, \hat{2}\}|=1$,

(iii) for $n\geq 1$, $\tau(\hat{S}_3^n)=3\tau(\hat{S}_3^{n-1})-1$,

(iv) $\tau(\hat{S}_3^n)=\frac {3^n+1} 2=\frac{|\hat{S}_3^n|} 3.$

\end{theorem}
\begin{proof} By induction on $n$. For convenience, let $a_n=\tau(\hat{S}_3^n)$ and $V_n=V(\hat{S}_3^n)$ for any integer $n\geq 0$.
If $n=0$, $\hat{S}_3^n\cong K_3$. Trivially, $a_0=1$, each minimum feedback vertex $\hat{S}_3^0$ consist of exactly one element of $\{\hat{0}, \hat{1}, \hat{2}\}$. Let $A_0=\{\hat{0}\}$.

\begin{figure}
\begin{center}
\scalebox {0.5}[0.5]{\includegraphics {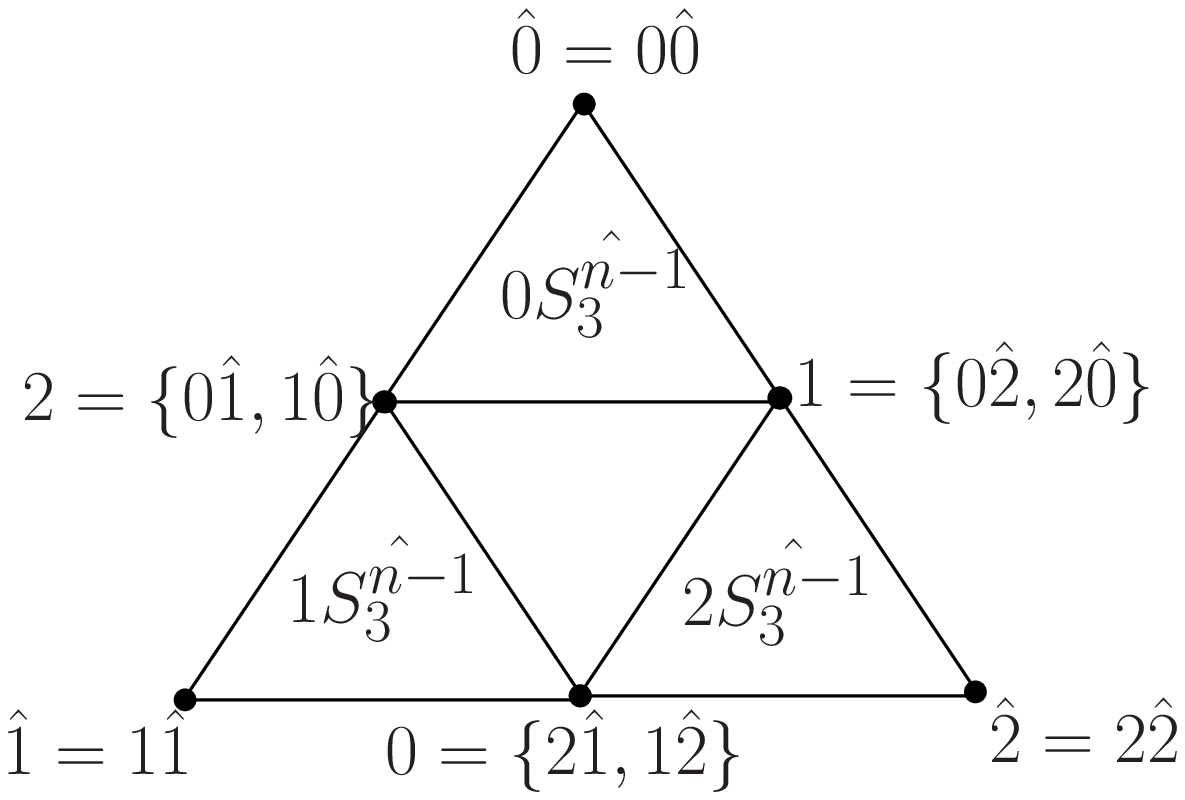} }
\centerline{Fig. 5\quad Relation between $\hat{S}_3^n$ and $iS_3^{n-1}$ for each $i\in T$}
\end{center}
\end{figure}

Now assume that $n\geq 1$ and the results holds for $n-1$.

First we prove (iii). By induction hypothesis and the symmetry of $\hat{S}_3^{n-1}$, there exists a minimum feedback vertex set  $A_{n-1}$  of $\hat{S}_3^{n-1}$ containing the vertex $\hat{0}$. To prove $a_n\leq 3a_{n-1}-1$,
let $V_n^i=\{ia|\ a\in V_{n-1}\}$ for each $i\in T$, where $0\hat{0}=\hat{0}$, $1\hat{1}=\hat{1}$, and $2\hat{2}=\hat{2}$; $0\hat{1}=2=1\hat{0}$,  $0\hat{2}=1=2\hat{0}$, and $2\hat{1}=0 = 1\hat{2}$ (see Fig. 5 for an illustration).

One can see that $V_n=V_n^0\cup V_n^1\cup V_n^2$ and $\hat{S}_3^n[V_n^j]\cong \hat{S}_3^{n-1}$ for each $j\in T$. For each $j\in T$, we define an isomorphism $f_j$ between $\hat{S}_3^{n-1}$ and $j\hat{S}_3^{n-1}$  as follows:

(1) $f_0(\hat{0})=\hat{0},  f_0(\hat{1})=2, f_0(\hat{2})=1$;

(2) $f_1(\hat{0})=0,  f_1(\hat{1})=2, f_1(\hat{2})=\hat{1}$;

(3) $f_2(\hat{0})=0,  f_2(\hat{1})=\hat{2}, f_2(\hat{2})=1$.

\vspace{2mm} Let $A_n=\bigcup_{j\in T}f_j(A_{n-1}).\qquad \qquad \qquad \qquad \qquad \qquad (*)$
\begin{figure}[!h]
\begin{center}
\scalebox {0.7}[0.7]{\includegraphics {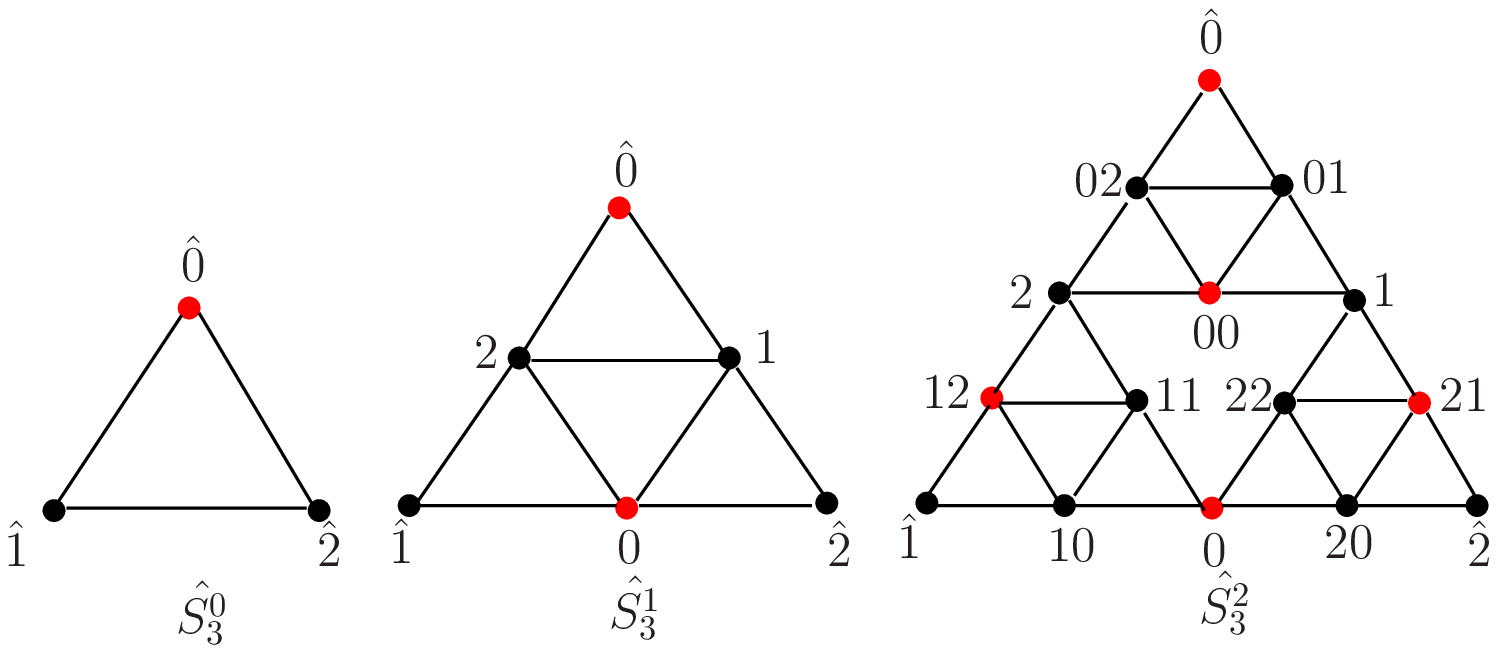} }
\centerline{Fig. 6\quad Minimum feedback vertex set $A_n$ (red vertices) of $\hat{S}_3^n$}
\centerline{ for $n\in\{0, 1, 2\}$ with the property (ii)}
\end{center}
\end{figure}

It can be checked that $A_n$ is a feedback vertex set of $\hat{S}_3^n$ with $$|A_n|=3a_{n-1}-1. \qquad   (**)$$ So, $a_n\leq 3a_{n-1}-1$.

Next we prove $a_n\geq 3a_{n-1}-1$. Let $A_n$ be a minimum feedback vertex set of $\hat{S}_3^{n}$, and let $A_n^i=A_n \cap i\hat{S}_3^{n-1}$ for each $i\in T$. Since $i\hat{S}_3^{n-1}\cong \hat{S}_3^{n-1}$, $A_n^i$ is a feedback vertex set of $i\hat{S}_3^{n-1}$. Hence, $|A_n^i|\geq a_{n-1}$. Note that $|A_n|=\sum_{i=0}^2|A_n^i|-|A_n\cap \{0, 1, 2\}|$.
If $|A_n\cap \{0, 1, 2\}|\leq 1$, then $a_n=|A_n|\geq 3a_{n-1}-1$.
Assume that $|\{0, 1, 2\}\cap A_n|=2$, and without loss of generality,
let $\{1, 2\}\subseteq A_n$. By induction hypothesis (i),
$|A_n^0|\geq a_{n-1}+1$. Hence $a_n\geq a_{n-1}+1+a_{n-1}-1+a_{n-1}-1=3a_{n-1}-1.$
If $|\{0, 1, 2\}\cap A_n|=3$, by induction hypothesis (i),
$|A_n^j|\geq a_{n-1}+1$ for each $j\in \{0, 1, 2\}$. Hence
$a_n\geq a_{n-1}+1+a_{n-1}+1+a_{n-1}+1-3=3a_{n-1}.$

This proves (iii).

\vspace{3mm} Let $A_n$ be the set constructed as $(*)$ in the proof of (iii). By $(**)$ and (iii), $A_n$ is a minimum feedback vertex set with $\hat{0}\in A_n$ and $\hat{1} \not\in A_n, \hat{2} \not\in A_n$ (for an illustration, see Fig. 6). This proves (ii).

\vspace{2mm} To prove (i), suppose that $X_n$ is a minimum vertex set of $\hat{S}_3^n$ with $|X_n\cap \{\hat{0}, \hat{1}, \hat{2}\}|\geq 2$.
Let $X_n^i=X_n \cap i\hat{S}_3^{n-1}$ for each $i\in T$. Since $i\hat{S}_3^{n-1}\cong \hat{S}_3^{n-1}$, $X_n^i$ is a feedback vertex set of $i\hat{S}_3^{n-1}$. Hence, $|X_n^i|\geq a_{n-1}$. Note that $|X_n|=\sum_{i=0}^2|X_n^i|-|X_n\cap \{0, 1, 2\}|$.

If $|\{0, 1, 2\}\cap X_n|=0$, $|X_n^j|\geq a_{n-1}$ for each $j\in \{0, 1, 2\}$.  Hence
$a_n\geq a_{n-1}+a_{n-1}+a_{n-1}=3a_{n-1}$, contradicting (iii). If $|\{0, 1, 2\}\cap X_n|=1$, by induction hypothesis (i), $|X_n^j|\geq a_{n-1}+1$ for some $j\in \{0, 1, 2\}$.
Thus $$a_n=|X_n|=\sum_{j=0}^2 |X_n^j|-1\geq (a_{n-1}+a_{n-1}+a_{n-1}+1)-1=3a_{n-1},$$ contradicting (iii).
If $|\{0, 1, 2\}\cap X_n|=2$, then at least two of $|X_n^1|, |X_n^2|, |X_n^3|$ are greater or equal to $a_{n-1}+1$.
Thus $a_n=|X_n|=\sum_{j=0}^2 |X_n^j|-2\geq (a_{n-1}+a_{n-1}+1+a_{n-1}+1)-2=3a_{n-1},$ contradicting (iii). If $|\{0, 1, 2\}\cap X_n|=3$, by induction hypothesis (i),
$|X_n^j|\geq a_{n-1}+1$ for each $j\in \{0, 1, 2\}$. Hence
$a_n\geq a_{n-1}+1+a_{n-1}+1+a_{n-1}+1-3=3a_{n-1}$, contradicting (iii). This proves (i).

\vspace{3mm}
Finally we prove (iv). Since $a_0=1$ and by (iii), we have $\tau(\hat{S}_3^n)=\frac {3^n+1} 2$. Moreover, by Lemma 3.1, $\tau(\hat{S}_3^n)=\frac{|\hat{S}_3^n|} 3.$
\end{proof}

\section{\large Generalized Sierpi\'{n}ski triangle graph $\hat{S_p^n}$ for $p\geq 4$ }

By Definition 2, for $p\in N$ and $n\in N_{0}$, the generalized Sierpi\'{n}ski triangle graph $\hat{S_p^n}$ is obtained by contracting all non-clique edges from the Sierpi\'{n}ski graph $S_p^{n+1}$. We denote the vertex obtained from contracting the edge $\{\underline{s}ij^{n-v+1}, \underline{s}j i^{n-v+1}\}\in E(S_p^{n+1} )$ by $\underline{s}\{i, j \}$. Then
$$V(\hat{S_p^n})=\hat{P}\cup \{\underline{s}\{i,j\}|\ \underline{s}\in P^{v-1}, v\in [n], \{i,j\}\in \big(^P_2 \big)                                                                                                                                                     \},$$ where $\hat{P}=\{\hat{k}\mid k\in P\}$, and
\begin{eqnarray*}
E(\hat{S_p^{n+1}})
&=&\{\{\hat{k},k^n\{j,k\}\}|\ k\in P,j\in P\setminus\{k\}\}\\
&\cup&\{\{\underline{s}\{i,j\},\underline{s}\{i,k\}\}|\ \underline{s}\in P^n, i\in P, \{j,k\}\in \big(^{P\setminus\{i\}}_{\ 2} \big)\}\\
&\cup& \{\{\underline{s}ki^{n-v}\{i,j\},\underline{s}\{i,k\}\}|\ \underline{s}\in P^{v-1},v\in[n],i\in P,j,k\in P\setminus\{i\}\}.
\end{eqnarray*}

%$\hat{S_1^n}$ is a vertice, $\hat{S_2^n}$ is a path.
The Sierpi\'{n}ski triangle graphs $\hat{S}_6^{2}$ and $\hat{S}_5^{2}$ are shown in Fig. 7 and  in Fig. 8, respectively. Trivially, $\hat{S_p^0}\cong K_p$. It is not hard to see that for $n\leq 2$, if $p$ is even,
\begin{figure}[!h]
\begin{center}
\scalebox {0.6}[0.6]{\includegraphics {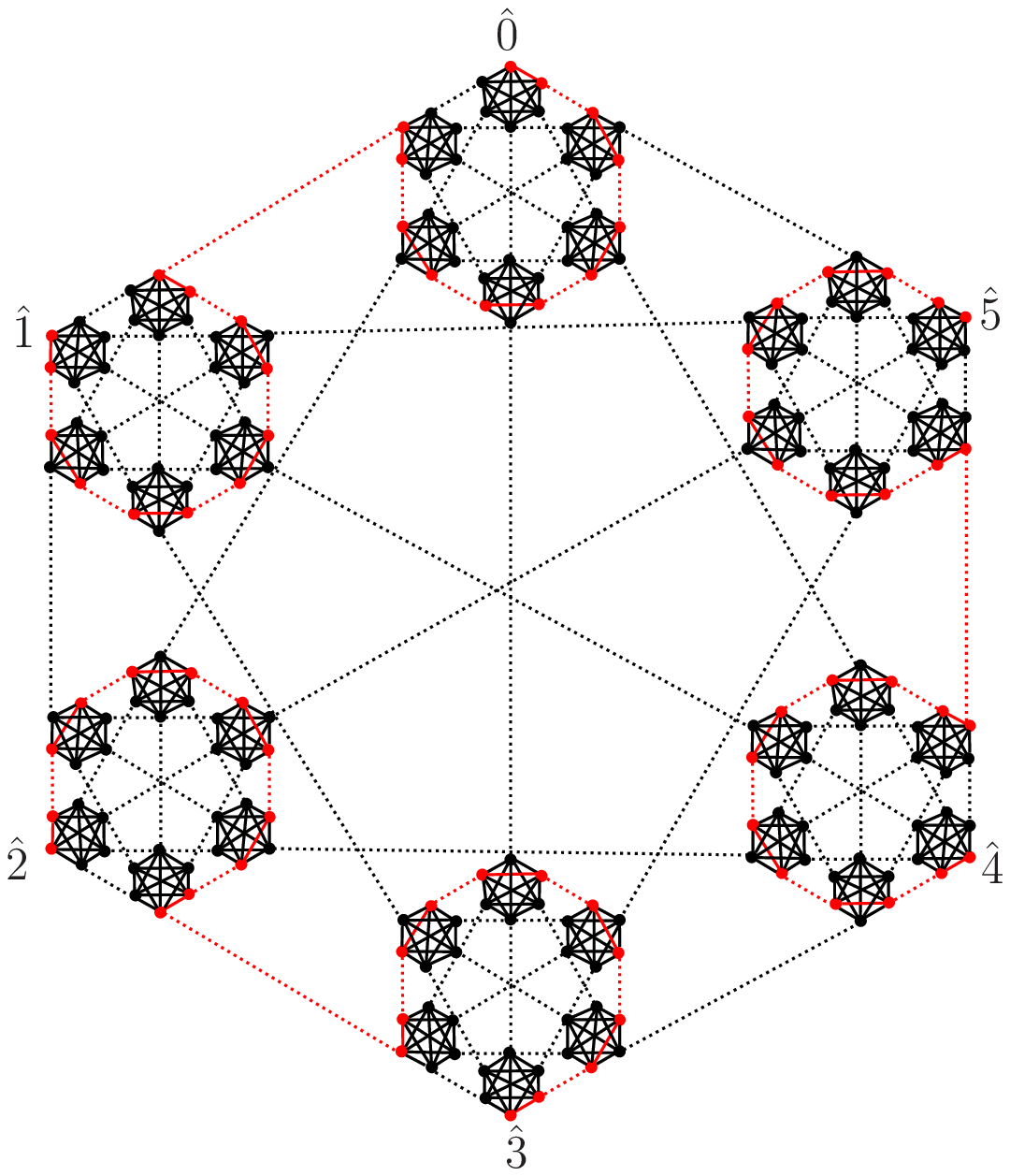}}
\centerline{Fig. 7\quad  $\hat{S_6^2}$, where a dotted edge of $S_6^3$ represents a vertex of $\hat{S_6^2}$ obtained from contracting it}
\centerline{  $\hat{S_6^2}[B_2^*]$ is the subgraph with colored in red }
\end{center}
\end{figure}

$$f(\hat{S_p^n})=\left \{
\begin{array}{ll}
2, & \mbox{if $n=0$ } \\
\frac{3p}{2}, & \mbox{if $n=1$ } \\
p^2+ \frac{p}{2}, & \mbox{if $n=2$}.
\end{array}
\right.
$$
If $p$ is odd,
$$f(\hat{S_p^n})=\left \{
\begin{array}{ll}
2, & \mbox{if $n=0$ } \\
\frac{3p-1}{2}, & \mbox{if $n=1$ } \\
p^2+\frac{p}{2}-\frac{1}{2}, & \mbox{if $n=2$}.
\end{array}
\right.
$$

In view of Section 3, it remains to consider the case when $p\geq 4$.

\begin{theorem} For two integers $n\geq 3$ and $p\geq 4$,
$$f(\hat{S_p^n})\geq \left \{
\begin{array}{ll}
p^n-\frac{p^{n-1}} 8 +\frac{p^{n-2}+\cdots+p} 8+\frac{5p} 8, & \mbox{if $p$ is even } \\
p^n-\frac{p^{n-1}+p^{n-2}-5p+3} 8, & \mbox{if $p$ is odd}.
\end{array}
\right.
$$
\end{theorem}

\begin{proof}
Let $V_n=V(\hat{S_p^n})$ and $iB_n^*=\{ib|\ b\in B_n^*\}$ for each $i\in P$, where $i \hat{i}=\hat{i}$ and $i\hat{j}=j\hat{i}=\{i, j\}$, $B_n^*\subseteq V_n$ will be defined recursively in terms of the parity of $p$ as follows.

\vspace{3mm}\noindent{\bf Case 1.} $p$ is even.

\vspace{2mm} Let $S=\{0,2,\cdots,p-2\}$ and $A=\{\{s_1,s_2\}|\ s_1, s_2\in S, s_1\neq s_2 \}$. Note that $A\subseteq V_n$.

For any $s\in S$, let $A_s^2=\{\hat{s},\hat{s+1},\{s, s+1\}, s\{i,i+1\},(s+1)\{i,i+1\}|\ i\in P\setminus \{s\}\}\subseteq V_2.$ It is easy to see that $\hat{S}_p^{2}[A_s^2]$ is path of length $2p$ connecting $\hat{s}$ and $\hat{s+1}$. For any $k\geq 3$, some subsets of $V_k$ defined by $A_s^k=sA_s^{k-1}\cup (s+1)A_s^{k-1}$, and further

$$B_k'=\bigcup_{s_1,s_2\in S, s_1\neq s_2 }(s_1A_{s_2}^{k-1}\cup (s_1+1)A_{s_2}^{k-1}\cup s_2A_{s_1}^{k-1}\cup (s_2+1)A_{s_1}^{k-1}),$$
$$B_k=\bigcup_{s_1, s_2\in S, s_1\neq s_2 }(s_1A_{s_2}^{k-1}\cup (s_1+1)A_{s_2}^{k-1}\cup s_2A_{s_1}^{k-1}\cup (s_2+1)A_{s_1}^{k-1})\setminus \{s_1, s_2\}),$$

$$ B_k^{0}=B_k, B_k^{n-k}=\{\underline{j}b|\ \underline{j}\in P^{n-k},b\in B_k\}.$$

By the definition above, one can see that $B_k'=B_k\setminus A$, $\bigcup_{j\in P}jB_k^{n-1-k}=B_k^{n-k}$, and  $A\cap (\bigcup_{s\in S}A_s^n)=\emptyset$, $A\cap B_k^{n-k}=\emptyset$. Let $B_2^*=\bigcup_{s\in S}A_s^2$. It can be seen that $\hat{S_p^2}[ B_2^*]$ is a linear forest (see $\hat{S_6^2}[ B_2^*]$ is depicted in red color in Fig. 7).

\vspace{2mm} \noindent{\bf Claim 1.} For any $k\geq 3$, let $B_k^*=(\bigcup_{j\in P}jB_{k-1}^*)\setminus A$. Then

\vspace{1mm} (i) $B_n^*=(\bigcup_{s\in S}A_s^n)\cup (\bigcup_{3\leq k\leq {n}}B_k^{n-k})$,

(ii) for each $s\in S$, $\hat{S_p^n}[A_{s}^{n}]$ is a path of order $2^{n-1}p+1$ joining $\hat{s}$ and $\hat{s+1}$,

(iii) for any $3\leq k\leq n$, $\hat{S_p^n}[B_k^{n-k}]$ is a linear forest consisting of $p^{n-k}\big(^{\frac{p}{2}}_2\big)$ paths of order $2^kp-1$,

(iv) $\hat{S_p^n}[B_n^*]$ is a linear forest.

\begin{proof} By induction on $n$. First we prove the statements hold for $n=3$.
Since $B_2^*=\bigcup_{s\in S}A_s^2$,
\begin{eqnarray*}
B_3^*
&=&(\bigcup_{j\in P}jB_{2}^*)\setminus A
=(\bigcup_{j\in P}j\bigcup_{s\in S}A_s^2)\setminus A\\
&=&(\bigcup_{s\in S}(sA_s^2\cup(s+1)A_s^2\cup(\bigcup_{j\in P\setminus\{s,s+1\}}jA_s^2)))\setminus A\\
&=&(\bigcup_{s\in S}(sA_s^2\cup (s+1)A_s^2)\cup(\bigcup_{s\in S}\bigcup_{j\in P\setminus\{s,s+1\}}jA_s^2))\setminus A\\
&=&(\bigcup_{s\in S}A_s^3)\cup B_3'\setminus A\\
&=&\bigcup_{s\in S}A_s^3\cup B_3,\\
\end{eqnarray*}
proving (i).

\vspace{2mm}
Now we prove (ii). Since $\hat{S_p^{3}}[A_s^3]=\hat{S_p^{3}}[sA_s^2\cup(s+1)A_s^2]$,  $\hat{S_p^{3}}[sA_s^2]\cong\hat{S_p^{3}}[(s+1)A_s^2]\cong\hat{S_p^{2}}[A_s^2]$. Recall that for a fixed $s\in S$, $\hat{S}_p^{2}[A_s^2]$ is path of order $2p+1$ connecting $\hat{s}$ and $\hat{s+1}$. It follows that
$\hat{S_p^{3}}[sA_s^2]$ and $\hat{S_p^{3}}[(s+1)A_s^2]$ are paths of order $2p+1$ joining $s\hat{s}$ and $s\hat{(s+1)}$, $(s+1)\hat{s}$ and $(s+1)\hat{(s+1)}$, respectively.
Moreover, since $s\hat{s}=\hat{s}$, $(s+1)\hat{(s+1)}=\hat{(s+1)}$ and $s\hat{(s+1)}=(s+1)\hat{s}=\{s, s+1\}$,
$\hat{S_p^{3}}[A_s^3]$ is a path connecting $\hat{s}$ and $\hat{s+1}$ of order $2(2p+1)-1=4p+1$.

\vspace{2mm} To prove (iii), let us consider a pair of element $s_1, s_2\in S$ with $s_1\neq s_2$. Since $$\hat{S_p^{3}}[s_1A_{s_2}^2]\cong\hat{S_p^{3}}[(s_1+1)A_{s_2}^2]\cong
\hat{S_p^{3}}[s_2A_{s_1}^2]\cong\hat{S_p^{3}}[(s_2+1)A_{s_1}^2]\cong\hat{S_p^{2}}[A_s^2],$$
$\hat{S_p^{3}}[s_1A_{s_2}^2]$, $\hat{S_p^{3}}[(s_1+1)A_{s_2}^2]$, $\hat{S_p^{3}}[s_2A_{s_1}^2]$, $\hat{S_p^{3}}[(s_2+1)A_{s_1}^2]$ are  paths of order $2p+1$ joining $s_1\hat{s_2}$ and $s_1\hat{(s_2+1)}$, $(s_1+1)\hat{s_2}$ and $(s_1+1)\hat{(s_2+1)}$, $s_2\hat{s_1}$ and $s_2\hat{(s_1+1)}$, $(s_2+1)\hat{s_1}$ and $(s_2+1)\hat{(s_1+1)}$, respectively. Moreover,
$s_1\hat{s_2}=s_2\hat{s_1}=\{s_1, s_2\}$, $s_1\hat{(s_2+1)}=(s_2+1)\hat{s_1}=\{s_1, s_2+1\}$, $s_2\hat{(s_1+1)}=(s_1+1)\hat{s_2}=\{s_1+1, s_2\}$, $(s_2+1)\hat{(s_1+1)}=(s_1+1)\hat{(s_2+1)}=\{s_1+1, s_2+1\}$. Hence $\hat{S_p^{3}}[s_1A_{s_2}^3\cup (s_1+1)A_{s_2}^3\cup s_2A_{s_1}^3\cup (s_2+1)A_{s_1}^3]$ is a cycle of order $4(2p+1)-4=8p$, and thus $\hat{S_p^{3}}[s_1A_{s_2}^3\cup (s_1+1)A_{s_2}^3\cup s_2A_{s_1}^3\cup (s_2+1)A_{s_1}^3]-\{s_1,s_2\}$ is a path of order $8p-1$.

\vspace{1mm} By the different choice for $s_1, s_2\in S$, we obtain total number of $\big(^{\frac{p}{2}}_2\big)$ such paths in $\hat{S_p^3}[B_3]$. This proves (iii).

\vspace{3mm} By (i), (ii), (iii), $\hat{S_p^3}[B_3^*]$ is a linear forest, proving (iv).

Now assume that $n\geq 4$ and the statements are true for smaller value of $n$.
Since $$B_{n-1}^*=(\bigcup_{s\in S}A_s^{n-1})\cup(\bigcup_{3\leq k\leq {n-1}}B_k^{n-1-k})$$ and $B_n^*=(\bigcup_{j\in P}jB_{n-1}^*)\setminus A$, we have
\vspace{-0.7cm}\begin{eqnarray*}
\\B_n^*
&=&(\bigcup_{j\in P}\bigcup_{s\in S}jA_s^{n-1})\cup(\bigcup_{j\in P}\bigcup_{3\leq k\leq {n-1}}jB_k^{n-1-k})\setminus A\\
&=&\bigcup_{s\in S}(sA_s^{n-1}\cup(s+1)A_s^{n-1}\cup(\bigcup_{j\in P\setminus\{s,s+1\}}jA_s^{n-1}))\cup(\bigcup_{3\leq k\leq {n-1}}B_k^{n-k})\setminus A\\
&=&\bigcup_{s\in S}(sA_s^{n-1}\cup(s+1)A_s^{n-1})\cup(\bigcup_{s\in S}\bigcup_{j\in P\setminus\{s,s+1\}}jA_s^{n-1})\cup(\bigcup_{3\leq k\leq {n-1}}B_k^{n-k})\setminus A\\
&=&\bigcup_{s\in S}A_s^n\cup (\bigcup_{3\leq k\leq {n-1}}B_k^{n-k})\cup B_n'\setminus A\\
&=&\bigcup_{s\in S}A_s^n\cup (\bigcup_{3\leq k\leq {n-1}}B_k^{n-k})\cup B_n\\
&=&\bigcup_{s\in S}A_s^n\cup (\bigcup_{3\leq k\leq {n}}B_k^{n-k}).\\
\end{eqnarray*}

This proves (i).

\vspace{3mm}
Now we prove (ii). By induction hypothesis (ii), $\hat{S}_p^{n-1}[A_s^{n-1}]$ is path connecting $\hat{s}$ and $\hat{s+1}$ of order $2^{n-2}p+1$.
Since $\hat{S_p^{n}}[sA_s^{n-1}]\cong\hat{S_p^{n}}[(s+1)A_s^{n-1}]\cong\hat{S_p^{n-1}}[A_s^{n-1}]$, $\hat{S_p^{n}}[A_s^n]$ is path connecting $\hat{s}$ and $\hat{s+1}$ of order $2(2^{n-2}p+1)-1=2^{n-1}p+1$ ( by the similar reason to that of case when $n=3$). This proves (ii).

\vspace{3mm} To prove (iii), let $3\leq k\leq {n-1}$. Since
$$B_k^{n-k}=\{jb|\ j\in P^{n-k},b\in B_k\}=\bigcup_{j\in p}jB_k^{n-1-k},$$ we have $$\hat{S_p^{n}}[B_k^{n-k}]=\hat{S_p^{n}}[\bigcup_{j\in p}jB_k^{n-1-k}].$$ For any $j\in P$, $\hat{S_p^{n}}[jB_k^{n-1-k}]\cong\hat{S_p^{n-1}}[B_k^{n-1-k}]$. By induction hypothesis (iii), there are $p^{n-1-k}\big(^{\frac{p}{2}}_2\big)$ paths of order $2^kp-1$ in $\hat{S_p^{n}}[jB_k^{n-1-k}]$. Moreover, since $E_{\hat{S_p^{n}}}[iB_k^{n-1-k},jB_k^{n-1-k}]=\emptyset$ for any $i, j\in P$ with $i\neq j$, there exists $pp^{n-1-k}\big(^{\frac{p}{2}}_2\big)$ paths of order $2^kp-1$ in $\hat{S_p^{n}}[B_k^{n-k}]$.

If $k=n$, $B_n^{n-n}=B_n=B_n'\setminus A$. Since $\hat{S_p^{n}}[s_1A_{s_2}^{n-1}\cup (s_1+1)A_{s_2}^{n-1}\cup s_2A_{s_1}^{n-1}\cup (s_2+1)A_{s_1}^{n-1}]$ is a cycle of order $4(2^{n-2}p+1)-4$ and $\hat{S_p^{n}}[s_1A_{s_2}^{n-1}\cup (s_1+1)A_{s_2}^{n-1}\cup s_2A_{s_1}^{n-1}\cup (s_2+1)A_{s_1}^{n-1}\setminus\{s_1,s_2\}]$ is a path of order $2^{n}p-1$ in $\hat{S_p^{n}}[B_n]$. By the different choice for $s_1, s_2\in S$, we obtain total number of $\big(^{\frac{p}{2}}_2\big)$ such paths in $\hat{S_p^n}[B_n]$. This proves (iii).

\vspace{3mm} By (i) (ii) (iii), we conclude that $\hat{S_p^n}[B_n^*]$ is a linear forest, proving (iv).

\end{proof}
\vspace{3mm} \noindent{\bf Claim 2.} For any $n\geq 3$,
\begin{equation}|B_n^*|=p\times |B_{n-1}^*|-\frac {p(p-1)} 2-\big(^{\frac{p}{2}}_2\big), \end{equation}
\begin{proof}
Since $B_n^*=\bigcup_{j\in P}jB_{n-1}^*\setminus A$ and $s_1\hat{s_2}=s_2\hat{s_1}=\{s_1, s_2\}$, we have $$|\bigcup_{j\in P}jB_{n-1}^*|=p\times |B_{n-1}^*|-\frac {p(p-1)} 2.$$ Moreover, since $|A|=\big(^{\frac{p}{2}}_2\big)$, $|B_n^*|=p\times |B_{n-1}^*|-\frac {p(p-1)} 2-\big(^{\frac{p}{2}}_2\big)$.
\end{proof}

 By Claim 1 and Claim 2, if $n\geq 3$, $\hat{S}_p^n$ has an induced forest of order $$p^n-\frac{p^{n-1}} 8 +\frac{p^{n-2}+\cdots+p} 8+\frac{5p} 8.$$

\noindent{\bf Case 2.} $p$ is odd.

\vspace{1mm} Let $S=\{0,2,\cdots,p-3\}$ and $A=\{\{s_1,s_2\}|\ s_1, s_2\in S, s_1\neq s_2 \}$.
For any $s\in S$, let $$A_s^2=\{\hat{s},\hat{s+1},\{s, s+1\}, s\{i,i+1\},(s+1)\{i,i+1\}|\ i\in P\setminus \{s\}\}\subseteq V_2.$$ It is easy to see that $\hat{S}_p^{2}[A_s^2]$ is path of length $2p$ connecting $\hat{s}$ and $\hat{s+1}$. For any $k\geq 3$, we define some subsets of $V_k$ as: $A_s^k=sA_s^{k-1}\cup (s+1)A_s^{k-1}$, and further

$$B_k'=\bigcup_{s_1,s_2\in S, s_1\neq s_2}(s_1A_{s_2}^{k-1}\cup (s_1+1)A_{s_2}^{k-1}\cup s_2A_{s_1}^{k-1}\cup (s_2+1)A_{s_1}^{k-1}),$$
$$B_k=\bigcup_{s_1,s_2\in S, s_1\neq s_2}(s_1A_{s_2}^{k-1}\cup (s_1+1)A_{s_2}^{k-1}\cup s_2A_{s_1}^{k-1}\cup (s_2+1)A_{s_1}^{k-1})\setminus \{s_1, s_2\}),$$
$$ B_k^{0}=B_k, B_k^{n-k}=\{\underline{j}b|\ \underline{j}\in P^{n-k},b\in B_k\}.$$

With the slight difference to the case when $p$ is parity, we define some more subsets of $V_k$ as follows.
$$T'=\{\hat{p-1},(p-1)\{i,i+1\}|\ i\in P-1\},$$
$$T_k=\bigcup_{s\in S }((p-1)A_s^{k-1}\cup s(p-1)^{k-3}T'\cup (s+1)(p-1)^{k-3}T'),$$

$$ T_k^{0}=T_k, T_k^{n-k}=\{\underline{j}t|\ \underline{j}\in P^{n-k},t\in T_k\}.$$
By the definition above, one can see that $B_k'=B_k\setminus A$, $\bigcup_{j\in P}jB_k^{n-1-k}=B_k^{n-k}$,\\ $\bigcup_{j\in P}jT_k^{n-1-k}=T_k^{n-k},$ $A\cap (\bigcup_{s\in S}A_s^n)=\emptyset$, $A\cap B_k^{n-k}=\emptyset$, $A\cap T_k^{n-k}=\emptyset$, $A\cap (p-1)^{n-2}T'=\emptyset$.

It is easy to see that $\hat{S}_p^{2}[T']$ is path of order $p$ connecting $(p-1)\{0,1\}$ and $\hat{p-1}$. Let $B_2^*=(\cup_{s\in S}A_s^2)\cup T'$. Trivially, $\hat{S_p^2}[ B_2^*]$ is a linear forest (see Fig. 8 for $\hat{S_5^2}[ B_2^*]$, as depicted in red color).

\begin{figure}[!h]
\begin{center}
\scalebox {0.5}[0.5]{\includegraphics {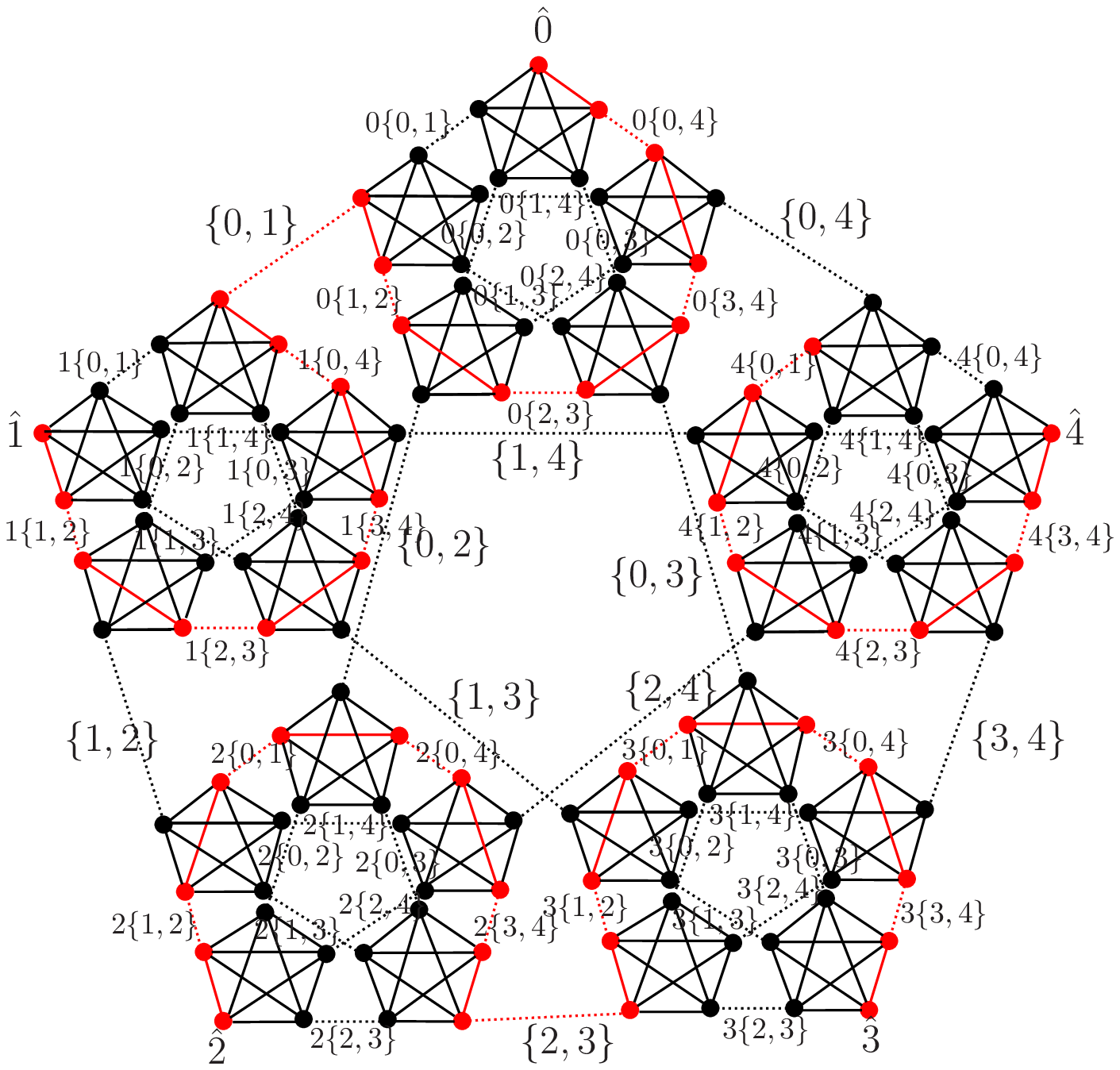}}
\centerline{Fig. 8\quad $\hat{S_5^2}$, where a dotted edge of $S_5^3$ represents a vertex of $\hat{S_5^2}$ result from contracting it}
\centerline{  $\hat{S_5^2}[B_2^*]$ is the subgraph with colored in red }
\end{center}
\end{figure}

\noindent{\bf Claim 3.}  For any $k\geq 3$, let $B_k^*=\bigcup_{j\in P}jB_{k-1}^*\setminus A$. Then

\vspace{2mm} (i) $B_n^*=(\bigcup_{s\in S}A_s^n)\cup (\bigcup_{3\leq k\leq {n}}B_k^{n-k})\cup (\bigcup_{3\leq k\leq {n}}T_k^{n-k})\cup (p-1)^{n-2}T'$,

(ii) for each $s\in S$, $\hat{S_p^n}[A_{s}^{n}]$ is a path of order $2^{n-1}p+1$ joining $\hat{s}$ and $\hat{s+1}$,

(iii) for any $3\leq k\leq n$, $\hat{S_p^n}[B_k^{n-k}]$ is a linear forest consisting of $p^{n-k}\big(^{\frac{p-1}{2}}_{\ 2}\big)$ paths of order $2^kp-1$,

(iv) for any $3\leq k\leq n$, $\hat{S_p^n}[T_k^{n-k}]$ is a linear forest consisting of $p^{n-k}\frac{p-1} 2$ paths of order $2^{k-2}p+2p-1$,

(v) $\hat{S_p^n}[(p-1)^{n-2}T']$ is a path of order $p$,

(vi) $\hat{S_p^n}[B_n^*]$ is a linear forest.

\begin{proof} By induction on $n$. First we prove the statements hold for $n=3$.
Since $B_2^*=(\bigcup_{s\in S}A_s^2)\cup T'$, we have
\vspace{-0.7cm}\begin{eqnarray*}
\\B_3^*
&=&(\bigcup_{j\in P}jB_{2}^*)\setminus A\\
&=&(\bigcup_{j\in P}j(\bigcup_{s\in S}A_s^2\cup T'))\setminus A\\
&=&((\bigcup_{j\in P\setminus\{p-1\}}j\bigcup_{s\in S}A_s^2)\setminus A)\cup(\bigcup_{j\in P}jT')\cup(\bigcup_{s\in S}(p-1)A_s^2)\\
&=&(\bigcup_{s\in S}A_s^3)\cup B_3\cup T_3\cup (p-1)T'.\\
\end{eqnarray*}
\vspace{3mm} The proof of (ii) and (iii) is completely similar to that of (ii) and (iii) for the case when $p$ is even. Next we prove (iv).
By the definition, $$T_3=\bigcup_{s\in S }((p-1)A_s^2\cup sT'\cup (s+1)T').$$ It can be seen that for a fixed $s\in S$,

$\hat{S_p^{3}}[(p-1)A_s^2]$ is a path of order $2p+1$ joining $(p-1)\hat{s}$ and $(p-1)\hat{(s+1)}$,

$\hat{S_p^{3}}[sT']$ is a path of order $p$ joining $s(p-1)\{0,1\}$ and $s\hat{(p-1)}$,

$\hat{S_p^{3}}[(s+1)T']$ is a path of order $p$ joining $(s+1)(p-1)\{0,1\}$ and $(s+1)\hat{(p-1)}$.

\noindent Moreover, since $(p-1)\hat{s}=s\hat{(p-1)}=\{s, p-1\}$, $(s+1)\hat{(p-1)}=(p-1)\hat{(s+1)}=\{s+1, p-1\}$,
$\hat{S_p^{3}}[(p-1)A_s^2\cup sT'\cup (s+1)T']$ is a path of order $2p+1+2p-2=4p-1$. So,
By the different choice for $s\in S$, we obtain total number of $\frac{p-1} 2$ such paths in $\hat{S_p^3}[T_3]$.

\vspace{3mm} Now we show (v). Since $\hat{S}_p^{2}[T']$ is path of order $p$ connecting $(p-1)\{0,1\}$ and $\hat{p-1}$ and $(p-1)\hat{p-1}=\hat{p-1}$, $\hat{S}_p^{3}[(p-1)^{3-2}T']$ is path of order $p$ connecting $(p-1)^{2}\{0,1\}$ and $\hat{p-1}$.

\vspace{3mm} By (i)-(v), we conclude that $\hat{S_p^3}[B_3^*]$ is a linear forest.

\vspace{3mm} Next assume that $n\geq 4$ and the statements are true for smaller value of $n$.
It is clear that $$\bigcup_{j\in P}\bigcup_{3\leq k\leq {n-1}}jB_k^{n-1-k}=\bigcup_{3\leq k\leq {n-1}}B_k^{n-k},$$ $$\bigcup_{j\in P}\bigcup_{3\leq k\leq {n-1}}jT_k^{n-1-k}=\bigcup_{3\leq k\leq {n-1}}T_k^{n-k}.$$
Since $p$ is odd, $p-1$ is even, we have $$(\bigcup_{j\in P\setminus\{p-1\}}\bigcup_{s\in S}j A_s^{n-2})\setminus A=(\bigcup_{s\in S}A_s^n)\cup B_n.$$
Since $$B_{n-1}^*=(\bigcup_{s\in S}A_s^{n-1})\cup (\bigcup_{3\leq k\leq {n-1}}B_k^{n-1-k})\cup (\bigcup_{3\leq k\leq {n-1}}T_k^{n-1-k})\cup (p-1)^{n-3}T'$$ and $B_n^*=\bigcup_{j\in P}jB_{n-1}^*\setminus A$, we have
\vspace{-0.7cm}\begin{eqnarray*}
\\B_n^*
&=&\bigcup_{j\in P}\bigcup_{s\in S}jA_s^{n-1}\cup\bigcup_{j\in P}\bigcup_{3\leq k\leq {n-1}}jB_k^{n-1-k}\cup\bigcup_{j\in P}\bigcup_{3\leq k\leq {n-1}}jT_k^{n-1-k}\cup(\bigcup_{j\in P}j(p-1)^{n-3}T')\setminus A\\
&=&(\bigcup_{j\in P\setminus\{p-1\}}\bigcup_{s\in S}jA_s^{n-1})\setminus A\cup\bigcup_{3\leq k\leq {n-1}}B_k^{n-k}\cup\bigcup_{3\leq k\leq {n-1}}T_k^{n-k}\cup\bigcup_{s\in S}(p-1)A_s^{n-1}\\
&&\cup\bigcup_{j\in P}j(p-1)^{n-3}T'\\
&=&\bigcup_{s\in S}A_s^n\cup B_n\cup\bigcup_{3\leq k\leq {n-1}}B_k^{n-k}\cup\bigcup_{3\leq k\leq {n-1}}T_k^{n-k}\cup\bigcup_{s\in S}(p-1)A_s^{n-1}\cup\bigcup_{j\in P\setminus\{p-1\}}j(p-1)^{n-3}T'\\
&&\cup(p-1)^{n-2}T'\\
&=&\bigcup_{s\in S}A_s^n\cup\bigcup_{3\leq k\leq n}B_k^{n-k}\cup\bigcup_{3\leq k\leq {n-1}}T_k^{n-k}\cup T_k\cup(p-1)^{n-2}T'\\
&=&\bigcup_{s\in S}A_s^n\cup\bigcup_{3\leq k\leq n}B_k^{n-k}\cup\bigcup_{3\leq k\leq n}T_k^{n-k}\cup(p-1)^{n-2}T'.\\
\end{eqnarray*}
proving (i).

\vspace{3mm} The proof of (ii) and (iii) is completely similar to that of (ii) and (iii) for the case when $p$ is even.

\vspace{3mm} To prove (iv), we first consider the case when $3\leq k\leq {n-1}$. Since
$$T_k^{n-k}=\{\underline{j}b|\ \underline{j}\in P^{n-k},b\in T_k\}=\bigcup_{j\in p}jT_k^{n-1-k},$$ we have $$\hat{S_p^{n}}[T_k^{n-k}]=\hat{S_p^{n}}[\bigcup_{j\in p}jT_k^{n-1-k}].$$ For any $j\in P$, $\hat{S_p^{n}}[jT_k^{n-1-k}]\cong\hat{S_p^{n-1}}[T_k^{n-1-k}]$. By induction hypothesis (iv), there are $p^{n-1-k}\frac{p-1} 2$ paths of order $2^{k-2}p+2p-1$ in $\hat{S_p^{n-1}}[jT_k^{n-1-k}]$. Moreover,\\ since $E_{\hat{S_p^{n}}}[iT_k^{n-1-k},jT_k^{n-1-k}]=\emptyset$ for any $i, j\in P$ with $i\neq j$,  there exists $p^{n-k}\frac{p-1} 2$ paths of order $2^{k-2}p+2p-1$ in $\hat{S_p^{n}}[T_k^{n-k}]$.

\vspace{3mm} If $k=n$, $T_n^{n-n}=T_n=\bigcup_{s\in S }((p-1)A_s^{n-1}\cup s(p-1)^{n-3}T'\cup (s+1)(p-1)^{n-3}T')$. Note that for a fixed $s\in S$,

$\hat{S_p^{n}}[(p-1)A_s^{n-1}]$ is a path of order $2^{n-2}p+1$ joining $(p-1)\hat{s}$ and $(p-1)\hat{(s+1)}$,

$\hat{S_p^{n}}[s(p-1)^{n-3}T']$ is a path of order $p$ joining $s(p-1)^{n-2}\{0,1\}$ and $s(p-1)^{n-3}\hat{(p-1)}=s\hat{(p-1)}$,

$\hat{S_p^{n}}[(s+1)(p-1)^{n-3}T']$ is a path of order $p$ joining $(s+1)(p-1)^{n-2}\{0,1\}$ and $(s+1)(p-1)^{n-3}\hat{(p-1)}=(s+1)\hat{(p-1)}$.

Furthermore, since $(p-1)\hat{s}=s\hat{(p-1)}=\{s, p-1\}$, $(s+1)\hat{(p-1)}=(p-1)\hat{(s+1)}=\{s+1, p-1\}$,
$\hat{S_p^{n}}[(p-1)A_s^{n-1}\cup s(p-1)^{n-3}T'\cup (s+1)(p-1)^{n-3}T']$ is a path of order $2^{n-2}p+2p-1$. By the different choice for $s\in S$, we obtain total number of $\frac{p-1} 2$ such paths in $\hat{S_p^n}[T_n]$. This prove (iv).

\vspace{3mm} Since $\hat{S}_p^{2}[T']$ is path of order $p$ connecting $(p-1)\{0,1\}$ and $\hat{p-1}$. Moreover, since  $(p-1)^{n-2}\hat{p-1}=\hat{p-1}$, $\hat{S}_p^{n}[(p-1)^{n-2}T']$ is path of order $p$ connecting $(p-1)^{n-1}\{0,1\}$ and $\hat{p-1}$, proving (v).

\vspace{3mm} By (i)-(v), $\hat{S_p^n}[B_n^*]$ is a linear forest. This proves (vi).

\end{proof}
\vspace{3mm} \noindent{\bf Claim 4.} For any $n\geq 3$,
\begin{equation}|B_n^*|=p\times |B_{n-1}^*|-\frac {p(p-1)} 2-\big(^{\frac{p-1}{2}}_{\  2}\big), \end{equation}
\begin{proof}
Since $B_n^*=\bigcup_{j\in P}jB_{n-1}^*\setminus A$ and $s_1\hat{s_2}=s_2\hat{s_1}=\{s_1, s_2\}$, we have $$|\bigcup_{j\in P}jB_{n-1}^*|=p\times |B_{n-1}^*|-\frac {p(p-1)} 2.$$ Since $|A|=\big(^{\frac{p-1}{2}}_{\  2}\big)$, $|B_n^*|=p\times |B_{n-1}^*|-\frac {p(p-1)} 2-\big(^{\frac{p-1}{2}}_{\  2}\big)$.
\end{proof}

By Claim 3 and Claim 4, if $n\geq 3$, $\hat{S}_p^n$ has an induced forest of order $$p^n-\frac{p^{n-1}+p^{n-2}-5p+3} 8. $$
\end{proof}

Since $\tau(G)+f(G)=|V(G)|$ for a graph $G$, Theorem 4.1 provides an upper bound for $\tau(\hat{S}_p^n)$. We suspect the upper bound is the exact value of feedback vertex number of the generalized Sierpi\'{n}ski triangle graph $\hat{S}_p^n$ for the case when $p\geq 4$.

\begin{conjecture}
For two integers $n\geq 3$ and $p\geq 4$,
$$f(\hat{S_p^n})=\left \{
\begin{array}{ll}
p^n-\frac{p^{n-1}} 8 +\frac{p^{n-2}+\cdots+p} 8+\frac{5p} 8, & \mbox{if $p$ is even } \\
p^n-\frac{p^{n-1}+p^{n-2}-5p+3} 8, & \mbox{if $p$ is odd}.
\end{array}
\right.
$$

\end{conjecture}

%\begin{figure}[!h]
%\begin{center}
%\scalebox {0.6}[0.6]{\includegraphics {fig8.eps}}
%\vskip -.3cm
%\small{Fig.8}
%\end{center}
%\end{figure}

%{\bf Definition 6.} For $p\in N$ and $n\in N_{0}$, the graph $^+\hat{S_p^{n}}$ is defined by
%\begin{equation*}V(^+\hat{S_p^{n}})=V(\hat{S_p^{n}})\cup\{w\},  E(^+\hat{S_p^{n}})=E(\hat{S_p^{n}})\cup\{\{w,i^n\}|i\in P\}.\end{equation*}
%
%{\bf Definition 6.} $^{++}\hat{S_p^{n}}$ is obtained from $^{++}{S_p^{n+1}}$ by contracting all non-clique edges.
%
%{\bf Corollary 3.} $\nu(^+\hat{S_p^{n}})=\nu(\hat{S_p^{n}}).$
%
%It is obviously $\hat{S_p^{n}}[{B_n}^{**}\cup{w}]$ contains cycle.
%
%{\bf Corollary 4.} when $p$ is odd,$\nu(^{++}\hat{S_p^{n}})=\nu(\hat{S_p^{n}})+\nu(\hat{S_p^{n-1}})-p,$. when $p$ is even. $\nu(^{++}\hat{S_p^{n}})=\nu(\hat{S_p^{n}})+\nu(\hat{S_p^{n-1}})-p-1.$

%\begin{proof}
%When $p$ is even,we can divided $^{++}\hat{S_p^{n}}$ into two parts $\hat{S_p^{n}}$ and $\hat{S_p^{n-1}}$. Let ${B_n}^{**}={B_{n1}}^{**}\cup{B_{n2}}^{**}$,${B_{n1}}^{**}\subseteq V(\hat{S_p^{n-1}}),$ ${B_{n2}}^{**}\subseteq V(\hat{S_p^{n}}).$ If ${B_{22}}^{**}={B_{2}}^{**}$,then   .The next discuss is similar to $\hat{S_p^{n}}$, $^{++}\hat{S_p^{n}}[{B_n}^{**}]$ will contains a cycle, see in Fig9. When $p$ is odd,the cycle will be path.Example see in fig9.
%\end{proof}
%\begin{figure}[!h]
%\begin{center}
%\scalebox {0.6}[0.6]{\includegraphics {fig9.eps}}
%\vskip -.3cm
%\small{Fig.9}
%\end{center}
%\end{figure}

\end{document}